\documentclass[11pt,twoside]{amsart}

\usepackage[
backend = biber,
citestyle = alphabetic,
style = numeric,
maxnames = 50,
giveninits=true,
]{biblatex}
\usepackage{amsmath,amsfonts,amssymb,amsthm}
\usepackage[all,cmtip]{xy}
\usepackage{tikz-cd}

\usepackage[colorlinks]{hyperref}
\hypersetup{
    colorlinks,
    citecolor=magenta,
    filecolor=red,
    linkcolor=blue,
    urlcolor=purple,
    pdftitle = On L-equivalence for K3 surfaces and hyperkähler manifolds
}
\usepackage{mathtools}

\usepackage{cleveref}
\usepackage[T2A]{fontenc}
\DeclareSymbolFont{cyrillic}{T2A}{cmr}{m}{n}
\DeclareMathSymbol{\Sha}{\mathalpha}{cyrillic}{216}
\usepackage{xcolor}
\renewcommand\thefootnote{\textcolor{red}{\arabic{footnote}}} 
\usepackage[normalem]{ulem}

\numberwithin{equation}{section}
\usepackage{microtype}
\newtheorem{dummy}{dummy}[section]

\newtheorem{definition}[dummy]{Definition}
\newtheorem{theorem}[dummy]{Theorem}
\newtheorem{corollary}[dummy]{Corollary}
\newtheorem{lemma}[dummy]{Lemma}
\newtheorem{proposition}[dummy]{Proposition}
\theoremstyle{definition}
\newtheorem{remark}[dummy]{Remark}
\newtheorem{example}[dummy]{Example}

\theoremstyle{plain}

\newtheorem{conjecture}[dummy]{Conjecture}

\newcommand{\A}{\ensuremath{\mathbb{A}}}
\newcommand{\Z}{\ensuremath{\mathbb{Z}}}
\newcommand{\R}{\ensuremath{\mathbb{R}}}

\newcommand{\Q}{\ensuremath{\mathbb{Q}}}
\newcommand{\CC}{\ensuremath{\mathbb{C}}}
\newcommand{\PP}{\ensuremath{\mathbb{P}}}
\newcommand{\LL}{\ensuremath{\mathbb{L}}}

\newcommand{\klat}{\ensuremath{\Lambda_{\operatorname{K3}}}}

\def\Db{\calD^{b}}

\newcommand{\gravk}{\ensuremath{K_0(\operatorname{Var}_k)}}
\newcommand{\gravc}{\ensuremath{K_0(\operatorname{Var}_\CC)}}
\newcommand{\gghs}{\ensuremath{K_0^\oplus(\operatorname{HS}_{\Z,2})}}
\newcommand{\gghsf}{\ensuremath{K_0^\oplus(\operatorname{HS}_{\Z})}}

\def\disc{\operatorname{disc}}
\def\Jac{\operatorname{J}}
\def\Kn{{\ensuremath{\operatorname{K3}^{[n]}}}}
\def\Ktwo{{\ensuremath{\operatorname{K3}^{[2]}}}}

\def\dv{\operatorname{div}}

\DeclareMathOperator{\Aut}{Aut}

\DeclareMathOperator{\rk}{rk}

\DeclareMathOperator{\Hom}{Hom}
  
\DeclareMathOperator{\End}{End}

\DeclareMathOperator{\id}{id}

\DeclareMathOperator{\SHom}{\mathcal{H\kern -1pt}\textit{om}} 
\DeclareMathOperator{\SEnd}{\mathcal{E\kern -1pt}\textit{nd}} 
\DeclareMathOperator{\SExt}{\mathcal{E\kern -1pt}\textit{xt}} 

\DeclareMathOperator{\NS}{NS}

\def\dar[#1]{\ar@<2pt>[#1]\ar@<-2pt>[#1]}

\newcommand\calD{\mathcal{D}}

\newcommand{\fonction}[5]{\begin{array}{lrcl} 
#1: & #2 & \longrightarrow & #3 \\
    & #4 & \longmapsto & #5 \end{array}}

\newcommand{\functionstar}[4]{
\begin{array}{rcl} #1 &\longrightarrow &#2 \\ #3&\longmapsto &#4 \end{array}
}

\newcommand{\isomorphismstar}[4]{
\begin{array}{rcl} #1 &\overset{\sim}{\longrightarrow} &#2 \\ #3&\longmapsto &#4 \end{array}
}
    
\title{On L-equivalence for K3 Surfaces and Hyperk\"ahler Manifolds}

\author[R.~Meinsma]{Reinder Meinsma}
\address{
Fakultät für Mathematik und Informatik,
Universität des Saarlandes,
Campus E2.4, 66123 Saarbrücken, Germany
}
\email{meinsma@math.uni-sb.de}

\begin{document}
\begin{abstract}
     This paper explores the relationship between $L$-equivalence and $D$-equivalence for K3 surfaces and hyperk\"ahler manifolds. Building on Efimov's approach using Hodge theory, we prove that very general $L$-equivalent K3 surfaces are $D$-equivalent, leveraging the Derived Torelli Theorem for K3 surfaces. Our main technical contribution is that two distinct lattice structures on an integral, irreducible Hodge structure are related by a rational endomorphism of the Hodge structure.
     We partially extend our results to hyperk\"ahler fourfolds and moduli spaces of sheaves on K3 surfaces.
\end{abstract}
\maketitle
\tableofcontents

\let\thefootnote\relax\footnotetext{The research in this work has been supported by the MIS (MIS/BEJ - F.4545.21) Grant from the FNRS, a mobility grant by the FNRS and an ACR Grant from the Universit\'e Libre de Bruxelles.}

\section{Introduction}
Two varieties $X,Y$ over a field $k$ are $L$-equivalent if there exists an integer $n\in \Z$ such that $$\LL^n\cdot ([X]-[Y]) = 0 \in \gravk.$$ Here, $\gravk$ denotes the Grothendieck ring of algebraic varieties over $k$. That is, $\gravk$ is generated, as an abelian group, by isomorphism classes $[X]$ of schemes of finite type over $k$, subject to the relations of the form $[X] = [X\setminus Z] + [Z]$, where $Z\subset X$ is any closed subset. The multiplication in $\gravk$ is given by the fibre product. The class of the affine line, denoted $\LL\coloneqq [\A^1_k]$, was proved to be a zero-divisor in $\gravk$ by Borisov \cite{Bor17}. This led to the notion of $L$-equivalence by Kuznetsov--Shinder \cite{KS18}. 

Two varieties $X$, $Y$ are called $D$-equivalent if there is a linear, exact equivalence of their bounded derived categories of coherent sheaves $\Db(X)\simeq \Db(Y)$. 

This work was originally motivated by a conjecture of Kuznetsov--Shinder predicting that $D$-equivalence should imply $L$-equivalence for simply connected varieties, c.f. \cite[Conjecture 1.6]{KS18}. This conjecture has since been disproved (see \cite{Mei25}). In light of this, a revised conjectural relationship is proposed in \cite{Mei25}: 

\begin{conjecture}\cite[Conjecture 1.5]{Mei25} \label{conj: L implies D}
    Let $X$ and $Y$ be projective hyperk\"ahler manifolds. If $X$ and $Y$ are $L$-equivalent, they are also $D$-equivalent, i.e.
    \[
    X\sim_\LL Y \implies \Db(X)\simeq \Db(Y).
    \]
\end{conjecture}

In this paper, we prove Conjecture \ref{conj: L implies D} in the case where $X$ and $Y$ are sufficiently general K3 surfaces. 
Throughout this paper, we work over the complex numbers.

\begin{theorem}[See Theorem \ref{thm: L equivalent implies D equivalent in general}] \label{thm: L equivalent implies D equivalent in general intro}
    Let $X$ be a K3 surface with $\rho\neq 18$ and $\End(T(X)_\Q)\simeq \Q$. Let $Y$ be a K3 surface.
    If $Y$ is $L$-equivalent to $X$, then $X$ and $Y$ are $D$-equivalent.
\end{theorem}

Theorem \ref{thm: L equivalent implies D equivalent in general intro} shows that $L$-equivalence implies $D$-equivalence for sufficiently general K3 surfaces of Picard rank $\rho\leq 17$, as well as for all K3 surfaces for which $\rk(T(X)) = 22-\rho$ is an odd prime. Indeed, the assumption that $\End(T(X)_\Q)\simeq \Q$ is satisfied by these K3 surfaces by \cite{SZ20,Ogu01,GV15} (see also Lemma \ref{lem: very general endomorphisms isometries}) and \cite[Remark 3.3.14]{Huy16}. 

Conjecture \ref{conj: L implies D} is still open for special K3 surfaces, as our methods rely on the assumption that $\End(T(X)_\Q)\simeq \Q$. However, the currently known examples of $L$-equivalent K3 surfaces have been constructed without relying on this assumption. Therefore, examples of $L$-equivalent special K3 surfaces are already known.

We also study $L$-equivalence for hyperk\"ahler manifolds of \Kn-type. The proof of Theorem \ref{thm: L equivalent implies D equivalent in general intro} relies on the Derived Torelli Theorem. A higher-dimensional generalisation of the Derived Torelli Theorem for hyperk\"ahler manifolds is currently an open problem. However, for special cases of hyperk\"ahler manifolds of \Kn-type, positive results can be found in the literature.  Using these results, we partially extend Theorem \ref{thm: L equivalent implies D equivalent in general intro} to  hyperk\"ahler manifolds. 
In particular, we study $L$-equivalence for hyperk\"ahler fourfolds of \Ktwo-type with Picard rank 1 using the results from \cite{KK24} (see Theorem \ref{thm: HK fourfolds}), and for moduli spaces of sheaves on K3 surfaces with unimodular N\'eron--Severi lattices using results of \cite{MM24} (see Theorem \ref{thm: unimodular moduli}).

We note that $L$-equivalence does not imply $D$-equivalence in general, even for simply connected varieties. For example, if $X$ and $Y$ are Fano varieties with the property that $[X] = [Y] \in \gravk$, then $X$ and $Y$ are trivially $L$-equivalent but not $D$-equivalent unless they are isomorphic, by \cite[Theorem 2.5]{BO01}. As a simple example of two such varieties, one may take $X = \mathrm{Bl}_p\PP^3$ for any point $p\in \PP^3$, and $Y = \PP^1 \times \PP^2$. In this case, one computes that $[X] = [Y] = \LL^3 + 2\LL^2 + 2\LL + 1$. 

Examples of $L$-equivalent K3 surfaces can be found in \cite{IMOU16,KS18,HL18,KKM20,SZ20}, and examples of $L$-equivalent hyperk\"ahler manifolds were constructed in \cite{Oka21}. 
In all these cases, the $L$-equivalent K3 surfaces and hyperk\"ahler manifolds are also $D$-equivalent. 

\subsection{Proof strategy}

Efimov approached $L$-equivalence using Hodge theory \cite{Efi18}, because $L$-equivalent complex smooth projective varieties $X,Y$ satisfy $[H^\bullet(X,\Z)] = [H^\bullet(Y,\Z)]$ in the Grothendieck group $\gghsf$ of integral, polarisable Hodge structures. 
With this approach, Efimov was able to show that $D$-equivalent abelian varieties need not be $L$-equivalent, and that $L$-equivalence classes of K3 surfaces are finite. 

Finally, Efimov showed that for $L$-equivalent K3 surfaces $X,Y$ which are sufficiently general, there exists an \textit{isomorphism} of integral Hodge structures $T(X)\simeq T(Y)$ \cite[Lemma 3.7]{Efi18}. Here, $T(X)$ denotes the transcendental sublattice of $H^2(X,\Z)$. 

Recall that the Derived Torelli Theorem for K3 surfaces, due to Mukai and Orlov, asserts that two K3 surfaces $X$, $Y$ are $D$-equivalent if and only if there is a Hodge \textit{isometry} $T(X)\simeq T(Y)$ \cite{Muk84,Orl03}. A Hodge isometry is an isomorphism of integral Hodge structures that respects the bilinear forms on $T(X)$ and $T(Y)$ given by the cup-product. 

Not every isomorphism of integral Hodge structures $T(X) \simeq T(Y)$ is a Hodge isometry. In Example \ref{ex:Shioda--Inose counterexample}, we give examples of K3 surfaces $X$, $Y$ with isomorphic Hodge structures $T(X)\simeq T(Y)$, yet no Hodge isometries exist between them. Therefore, \cite[Lemma 3.7]{Efi18} by itself is insufficient to establish that $L$-equivalent K3 surfaces are $D$-equivalent.

We study the relationship between Hodge isomorphisms and Hodge isometries for integral, irreducible Hodge structures with a bilinear form, which we call \textit{Hodge lattices}. Our main technical result is the following:

\begin{proposition}[see Proposition \ref{prop: F H bijection} and Lemma \ref{lem: isomorphic hodge structure means twisted by endomorphism}] \label{prop: main technical contribution intro}
    Let $T_1$ and $T_2$ be two irreducible Hodge lattices of K3 type. If $T_1$ and $T_2$ are isomorphic as Hodge structures, then there exists a rational Hodge automorphism $\phi\colon (T_1)_\Q\simeq (T_1)_\Q$ with the property that $T_2$ is Hodge isometric to $(T_1)_\phi$, which is the Hodge lattice with underlying Hodge structure $T_1$ and lattice structure given by $(x\cdot y)_{(T_1)_\phi} = (\phi(x)\cdot y)_{T_1}$ for all $x,y\in T_1$.
\end{proposition}
We note that for the bilinear form $(\phi(x)\cdot y)_{T_1}$ to be a symmetric bilinear form that takes values in the integers, the rational Hodge endomorphism $\phi$ must satisfy two conditions. Firstly, the form is symmetric if and only if $\phi$ is invariant under the Rosati involution, that is, we have $(\phi(x)\cdot y)_{T_1} = (x\cdot \phi(y))_{T_1}$ for all $x,y\in T_1$. Secondly, the form takes values in the integers if and only if $\phi$ maps $T_1$ into its dual $(T_1)^*$. 

A sufficiently general K3 surface $X$ has very few rational Hodge endomorphisms, that is, we have $\End(T(X)_\Q)\simeq \Q$. Therefore, if $X$ is $L$-equivalent to a K3 surface $Y$, Proposition \ref{prop: main technical contribution intro} combined with \cite[Lemma 3.7]{Efi18} asserts that there is a Hodge isometry $T(Y) \simeq T(X)_q$ for some $q\in \Q$. Therefore, we can compute the discriminant of $T(Y)$ via $$\disc(T(Y)) = \disc(T(X)_q) = q^{\rk(T(X))}\disc(T(X))$$ (see Lemma \ref{lem: discriminants of twisted hodge lattices}). Theorem \ref{thm: L equivalent implies D equivalent in general intro} follows from the above observations combined with the following lemma.
\begin{lemma}[See Lemma \ref{lem:L equivalent same discriminant}] \label{lem: L equivalent same discriminant intro}
    If $X$ and $Y$ are $L$-equivalent K3 surfaces, then we have $\disc(T(X)) = \disc(T(Y))$.
\end{lemma}

\subsection{Structure of the paper}

In Section \ref{sec: lattices and Hodge structures}, we recall the necessary lattice and Hodge theory. The main goal of this section is to explain Definition \ref{def:twist hodge lattice by endomorphism}, which is the definition of a Hodge lattice twisted by a rational Hodge endomorphism. 

In Section \ref{sec: Hodge isomorphisms and Hodge isometries}, we introduce the notion of \emph{T}-equivalence: we say that K3 surfaces $X$ and $Y$ are $T$-equivalent if the Hodge structures $T(X)$ and $T(Y)$ are isomorphic, but not necessarily isometric. We investigate the relationship between $T$-equivalence and $D$-equivalence for K3 surfaces. There are two main results in this section, namely Corollary \ref{cor: main technical contribution} and Proposition \ref{prop: T general then T implies D}. Corollary \ref{cor: main technical contribution} is a restatement of Proposition \ref{prop: main technical contribution intro} above, and Proposition \ref{prop: T general then T implies D} shows that $T$-equivalence implies $D$-equivalence under the assumption that $\disc(T(X)) = \disc(T(Y))$.

In Section \ref{sec: L equivalence and D equivalence}, the main goal is to use Proposition \ref{prop: T general then T implies D} to prove that sufficiently general $L$-equivalent K3 surfaces are $D$-equivalent. This proof consists of verifying the condition of Proposition \ref{prop: T general then T implies D} that $L$-equivalent K3 surfaces $X$ and $Y$ satisfy $\disc(T(X)) = \disc(T(Y))$. We do this by defining a group homomorphism $\gravc \to \Q^\times$, which factorises through $\gravc[L^{-1}]$ and sends the class of a K3 surface $[X]$ to $\disc(T(X))$. 

Finally, in Section \ref{sec: applications and challenges}, we use the techniques developed in this paper to study hyperk\"ahler manifolds of \Kn-type, and we discuss the challenge that our methods face in the higher-dimensional setting: the question of whether Lemma \ref{lem: L equivalent same discriminant intro} generalises to the higher-dimensional setting. 

\subsection*{Acknowledgements}
I would like to thank my PhD advisor Evgeny Shinder for the original idea of using Efimov's work to prove Theorem \ref{thm: L equivalent implies D equivalent in general}, and for many useful conversations and comments on an earlier draft of this paper. I also want to thank Mauro Varesco for his very helpful comments, and Yulieth Prieto-Montañez for pointing out a typo in the first version of this paper.

\section{Lattices and Hodge structures} \label{sec: lattices and Hodge structures}
\subsection{Lattices}
Our main reference for lattice theory is \cite{Nik80}. A lattice is a free, finitely generated abelian group $L$ together with a non-degenerate symmetric bilinear form $b\colon L\times L\to \Z$. For $v,w\in L$, we usually denote $v\cdot w\coloneqq b(v,w)$ and sometimes $(v\cdot w)_L\coloneqq b(v,w)$. A group isomorphism of lattices respecting the bilinear forms is called an \textit{isometry}. The group of isometries of $L$ is denoted $O(L)$. A lattice $L$ is called \textit{even} if $v^2\coloneqq v\cdot v$ is even for all $v\in L$. We assume all our lattices to be even.
We define the \textit{dual lattice} to be 
$L^*\coloneqq \Hom(L,\Z)$. The dual lattice inherits a (usually not integral) symmetric bilinear form from $L$. There is a natural embedding $L^*\hookrightarrow L_\Q\coloneqq L\otimes \Q$ that respects the bilinear forms and whose image is 
\begin{equation}\label{eq:dual lattice in rational}
    \left\{x\in L_\Q\mid x\cdot v\in \Z \text{ for all } v\in L\right\}.
\end{equation}

Since we assume $L$ to be non-degenerate, the natural map $L\hookrightarrow L^*$ is injective.
The \textit{discriminant lattice} of $L$ is the quotient $$A_L\coloneqq L^*/L.$$ 
It has a natural quadratic form $q\colon A_L\to \Q/2\Z$ which it inherits from $L^*.$
The \textit{discriminant} of $L$ is $$\disc(L)\coloneqq |A_L|.$$ If we choose a $\Z$-basis for $L$, and let $M$ be the integral matrix representing $b$ in this basis, then we have $$\disc(L) = |\det(M)|.$$ 

For an integer $n\in \Z$, we write $L(n)$ for the lattice whose underlying group structure is that of $L$, and whose bilinear form is given by $(v\cdot w)_{L(n)} = (v\cdot w)_L.$ Note that we have $$\disc(L(n)) = |n|^{\rk(L)}\disc(L).$$

We say that $L$ is \textit{unimodular} if $\disc(L)=1$. An important example of a unimodular lattice is the hyperbolic plane $U$, which is the lattice of rank 2 whose bilinear form is given by $$\left(\begin{matrix}
    0&1\\1&0
\end{matrix}\right).$$ Another important example is the lattice $E_8$, which is the unique even unimodular positive-definite lattice of rank 8.
 
\subsection{K3 surfaces and hyperk\"ahler manifolds}
Our basic reference on K3 surfaces is \cite{Huy16}. We assume all our K3 surfaces and hyperk\"ahler manifolds to be projective.
Let $X$ be a K3 surface. The integral cohomology group $H^2(X,\Z)$ is a free abelian group of rank 22. The cup-product is a non-degenerate symmetric bilinear form $H^2(X,\Z)\times H^2(X,\Z)\to \Z$ which turns $H^2(X,\Z)$ into a unimodular lattice. There is an isometry $H^2(X,\Z)\simeq \klat,$ where $$\klat\coloneqq U^{\oplus 3}\oplus E_8(-1)^{\oplus 2}$$ is the \textit{K3-lattice}.

The cohomology group $H^2(X,\Z)$ has a natural Hodge structure of weight 2:
$$H^2(X,\CC)\simeq H^{0,2}(X)\oplus H^{1,1}(X)\oplus H^{2,0}(X),$$ and $H^{2,0}(X)$ is generated by an everywhere non-degenerate symplectic form $\sigma\in H^{2,0}(X)$.
The \textit{N\'eron--Severi lattice} of $X$ is the sublattice $$\NS(X)\coloneqq H^{1,1}(X)\cap H^2(X,\Z).$$ The rank of $\NS(X)$ is called the \textit{Picard rank} of $X$, usually denoted $\rho$. The \textit{transcendental lattice} of $X$ is $$T(X)\coloneqq \NS(X)^\perp\subset H^2(X,\Z).$$ Note that we have $H^{2,0}(X)\subset T(X)_\CC$, and in fact $T(X)$ is the smallest primitive sublattice of $H^2(X,\Z)$ with this property. The rank of $T(X)$ is $22-\rho$.

\begin{definition}
\begin{enumerate}
    \item[i)] A lattice with a Hodge structure is called a \textit{Hodge lattice}.
    \item[ii)] A Hodge structure $H$ of weight 2 is said to be of K3-type if $H^{2,0}$ is a $1$-dimensional $\CC$-vector space. 
    \item[iii)] A Hodge structure of K3-type $T$ is irreducible if $T$ admits no proper, primitive sub-Hodge structure of K3-type.
\end{enumerate}
\end{definition}

By the above, the transcendental lattice of a K3 surface is an irreducible Hodge lattice of K3-type.

More generally, if $X$ is a complex projective hyperk\"ahler manifold, the cohomology group $H^2(X,\Z)$ admits a natural integral symmetric bilinear form called the \textit{Beauville--Bogomolov--Fujiki} (BBF) form \cite{Bea83,Bog96,Fuj87}. The BBF form turns $H^2(X,\Z)$ into a non-degenerate integral Hodge lattice of K3-type. Similarly to the case of K3 surfaces, we write $\NS(X)\coloneqq H^{1,1}(X)\cap H^2(X,\Z)$ for the N\'eron--Severi lattice, and we let $T(X) \coloneqq \NS(X)^{\perp}\subset H^2(X,\Z)$ be the transcendental lattice of $X$. The transcendental lattice is an irreducible Hodge lattice of K3-type.

\begin{remark}\label{rem: isomorphism diagram that commutes iff isometry}
    Let $T$ be a Hodge lattice. Then $T^*$ inherits a Hodge structure from $T_\Q$ via the natural inclusion $T^*\subset T_\Q$ (see \eqref{eq:dual lattice in rational}). With this Hodge structure, the natural embedding $T\hookrightarrow T^*$ is a morphism of Hodge structures.

    Suppose $T_1$ and $T_2$ are Hodge lattices, and that $f\colon T_1\simeq T_2$ is an isomorphism of Hodge structures. Then $f$ induces an isomorphism of Hodge structures $f^*\colon T_2^*\simeq T_1^*$. However, the diagram 
    \begin{equation}\label{eq: isomorphism diagram that commutes iff isometry}
        \xymatrix{
            T_1 \ar[r]^{f}_\simeq \ar@{^(->}[d]_{i_{1}} & T_2 \ar@{^(->}[d]^{i_{2}} \\ 
            T_1^* & T_2^*\ar[l]^{f^*}_\simeq
        }
    \end{equation}
    is not necessarily commutative. In fact, \eqref{eq: isomorphism diagram that commutes iff isometry} commutes if and only if $f$ is an isometry.
\end{remark}

\subsection{Irreducible Hodge lattices}
Let $T$ be an integral irreducible Hodge lattice of K3-type. Consider the Hodge endomorphism algebra $E\coloneqq \End(T_\Q)$. Elements of $E$ will be called \textit{rational Hodge endomorphisms}. Note that an isomorphism of integral Hodge structures $\phi \colon T\simeq T$ induces a rational Hodge endomorphism $\phi_\Q\in E$. The converse does not hold: a rational Hodge endomorphism does not necessarily preserve $T\subset T_\Q$.

Since $T$ is irreducible, it follows that $E$ is a number field. Moreover, it is well-known that $E$ is either totally real or a CM field, see for example \cite[Theorem 3.3.7]{Huy16} for details. Recall that a totally real number field is a number field $k$ for which all embeddings $k\hookrightarrow\CC$ have images contained in $\R\subset \CC,$ and a CM field is a number field of the form $k(\sqrt{a})$, where $k$ is totally real and $\iota(a)<0$ for all embeddings $\iota\colon k\hookrightarrow \R$.

For a Hodge endomorphism $\phi\in E$, we define $\overline{\phi}\in E$ to be the unique Hodge endomorphism with the property that $$\phi(x)\cdot y = x\cdot \overline{\phi}(y)$$ for all $x,y\in T_\Q$. The group homomorphism $$\functionstar{E}{E}{\phi}{\overline{\phi}}$$ is known as the \textit{Rosati involution}, and we consider the invariant subfield $$K(T)\coloneqq \left\{\phi\in E\mid \phi = \overline{\phi}\right\}.$$
We note that $E = K(T)$ if $E$ is totally real, and $E = K(T)(\sqrt{a})$ for some $a\in K(T)$ otherwise.

\begin{definition}\label{def:twist hodge lattice by endomorphism}
    Let $T$ be an integral irreducible Hodge lattice of K3-type, and let $\phi\in K(T)$ be a rational Hodge endomorphism which maps $T$ into $T^*$. We write $T_\phi$ for the integral irreducible Hodge lattice of K3-type whose underlying Hodge structure is equal to that of $T$ and whose lattice structure is given by 
    \begin{equation}\label{eq:hodgelattice twisted by endomorphism}
    (x\cdot y)_{T_\phi} = (\phi(x)\cdot y)_T
    \end{equation}
    for all $x,y\in T$.
\end{definition}

\begin{remark}\label{rem: basics of twisting by endomorphism}
    \begin{itemize}
        \item[i)] Note that \eqref{eq:hodgelattice twisted by endomorphism} is a symmetric bilinear form by our assumption that $\phi\in K(T)$, since we have $$(x\cdot y)_{T_\phi} = (\phi(x)\cdot y)_T =(x\cdot\phi(y))_T = (y\cdot x)_{T_{\phi}}$$ for all $x,y\in T$.
        \item[ii)] By \eqref{eq:dual lattice in rational}, the assumption that $\phi$ maps $T$ into $T^*$ is equivalent to \eqref{eq:hodgelattice twisted by endomorphism} taking values in the integers. 
    \end{itemize}
\end{remark}

Recall that for a number field $k$ and an element $a\in k$, the \textit{norm} of $a$ is defined to be $$N(a)\coloneqq \prod_{\iota} \iota(a),$$ where the product runs over all embeddings $\iota\colon k\hookrightarrow \CC$.

\begin{lemma} \label{lem: discriminants of twisted hodge lattices} \cite[Lemma 4.5(2)]{Gee06}
    Let $T$ be an integral irreducible Hodge lattice of K3-type. Let $E\coloneqq \End(T_\Q)$ be the rational endomorphism field of $T$, and let $\phi\in K(T)$ be a rational Hodge endomorphism which maps $T$ into $T^*$. Then we have $$\disc(T_\phi) = N(\phi)^m\disc(T),$$ where $m \coloneqq \dim_E(T_\Q).$
\end{lemma}

\section{Hodge isomorphisms and Hodge isometries}
\label{sec: Hodge isomorphisms and Hodge isometries}
For a Hodge lattice $T$, we denote the group of Hodge isometries of $T$ by $O_{\operatorname{Hodge}}(T)$. Similarly, we write $O_\mathrm{Hodge}(T_\Q)$ for the group of rational Hodge automorphisms which respect the natural bilinear form on $T_\Q$, which we call \textit{rational Hodge isometries}. 

\begin{lemma} \label{lem: very general endomorphisms isometries}
    Let $X$ be a K3 surface. Consider the four statements:
    \begin{enumerate}
        \item[i)] $\End(T(X)) \simeq \Z$,
        \item[ii)] $\End(T(X)_\Q)\simeq \Q$, 
        \item[iii)] $O_{\operatorname{Hodge}}(T(X)_\Q)\simeq \Z/2\Z$, 
        \item[iv)] $O_{\operatorname{Hodge}}(T(X))\simeq \Z/2\Z$. 
    \end{enumerate}
    Then we have $$i)\iff ii) \implies iii)\implies iv).$$ Moreover, if $X$ has Picard rank $\rho\leq 19$ and $X$ is very general in a moduli space of lattice-polarised K3 surfaces, then $X$ satisfies each of these statements.
\end{lemma}
\begin{proof}
    The equivalence of $i)$ and $ii)$ follows from the observation that $\End(T(X))\otimes \Q\simeq \End(T(X)_\Q)$. 
    
    $ii)\implies iii)$:
    If we assume $\End(T(X)_\Q)\simeq \Q$, then $K(T(X)) = \End(T(X))$. Therefore, for any rational endomorphism $\phi\in \End(T(X)_\Q)$, we have $\phi(x)\cdot \phi(y) = x\cdot \phi^2(y)$. If $\phi$ is an isometry, this means that we have $\phi^2(y) = y$ for all $y\in T(X)$, therefore $\phi^2 = \id_{T(X)_\Q}$. As $\End(T(X)_\Q)\simeq \Q$, this implies $\phi = \pm \id_{T(X)_\Q}$, hence $O_{\operatorname{Hodge}}(T(X)_\Q)\simeq \Z/2\Z$. 

    $iii)\implies iv)$:    
    For any integral Hodge isometry $\phi\in O_{\operatorname{Hodge}}(T(X))$, the rational Hodge endomorphism $\phi_\Q$ is a rational Hodge isometry. Moreover, for two integral Hodge isometries $\phi, \psi\in O_\mathrm{Hodge}(T(X))$, we have $\phi_\Q = \psi_\Q$ if and only if $\phi = \psi$. Thus, we have an injective group homomorphism 
    \begin{equation} \label{eq: injective rationalisation homomorphism}
        O_\mathrm{Hodge}(T(X)) \hookrightarrow O_\mathrm{Hodge}(T(X)_\Q)
    \end{equation}
    Since $O_\mathrm{Hodge}(T(X)_\Q) \simeq \Z/2\Z$, it follows that $O_\mathrm{Hodge}(T(X))$ is either trivial, or \eqref{eq: injective rationalisation homomorphism} is an isomorphism. Since $\pm \id_{T(X)} \in O_\mathrm{Hodge}(T(X))$ are two distinct integral Hodge isometries, \eqref{eq: injective rationalisation homomorphism} must be an isomorphism, as required.

    The final claim is well-known, see \cite[Lemma 3.9]{SZ20} and \cite[Lemma 4.1]{Ogu01} for a proof that a very general K3 surface satisfies iii) and iv), and \cite[Lemma 9]{GV15} for i) and ii).
\end{proof}

\begin{definition}
    \begin{enumerate}
        \item Two hyperk\"ahler manifolds $X, Y$, of the same dimension, are \textit{$T$-equivalent} if there is an isomorphism of Hodge structures $T(X)\simeq T(Y)$. 
        \item Two hyperk\"ahler manifolds $X$ and $Y$ are \textit{$D$-equivalent} if there is an equivalence $\Db(X)\simeq \Db(Y)$.
    \end{enumerate}
\end{definition}

For hyperk\"ahler manifolds of \Kn-type, an equivalence $\Db(X)\simeq \Db(Y)$ induces a Hodge isometry $T(X)\simeq T(Y)$ \cite[Corollary 9.3]{Bec22}. In particular, for these, $D$-equivalence implies $T$-equivalence. 
In this section, we investigate the converse of this statement.

\begin{example}\label{ex:Shioda--Inose counterexample}
The following two examples show that $T$-equivalence does not imply $D$-equivalence.
\begin{itemize}
    \item[i)]Let $X$ be a K3 surface admitting a Shioda--Inose structure \cite{Mor84}. Then there is a K3 surface $Y$ for which there is a Hodge isometry $T(X)\simeq T(Y)(2)$. This Hodge isometry gives an isomorphism of Hodge structures $T(X)\simeq T(Y)$. However, we have $$\operatorname{disc}(T(X)) = \operatorname{disc}(T(Y)(2)) \neq \operatorname{disc}(T(Y)),$$ hence there is no Hodge isometry between $T(X)$ and $T(Y)$.

    \item[ii)] Let $X$ be a K3 surface with Picard rank $\rho\geq 12$ and let $n>0$ be any positive integer. Then there exists a K3 surface $Y$ with $T(Y)\simeq T(X)(n)$. Indeed, since $\rk(T(X))\leq 10$, there exists a primitive embedding into the K3-lattice $T(X)(n)\hookrightarrow \klat$ by \cite[Corollary 1.12.3]{Nik80}. By the surjectivity of the period map \cite{Tod80}, there exists a K3 surface $Y$ for which there exists a Hodge isometry $T(Y)\simeq T(X)(n)$.
\end{itemize}
\end{example}

Example \ref{ex:Shioda--Inose counterexample} ii) shows that $T$-equivalence classes are not necessarily finite. In our counterexamples, the discriminants of $T$-equivalent K3 surfaces can get arbitrarily large. Once we fix the discriminant, however, we have the following result of Efimov.

\begin{proposition}\cite[Proposition 3.3]{Efi18}
    Let $X$ be a K3 surface, and let $d\in \Z$ be an integer. Then the set of isomorphism classes of K3 surfaces $Y$, that are $T$-equivalent to $X$ and satisfy $\disc(T(Y))=d$, is finite.
\end{proposition}

\begin{definition} \label{def: F H}
Let $T$ be an integral irreducible Hodge lattice of K3-type.
\begin{enumerate}
    \item[i)] We denote by $$F(T)\subset K(T)^\times$$ the set of rational Hodge automorphisms $\phi\colon T_\Q\simeq T_\Q$ that are contained in $K(T)^\times$ and that map $T$ into $T^*$. In other words, by Remark \ref{rem: basics of twisting by endomorphism}, $F(T)$ is the set of rational Hodge automorphisms $\phi$ for which $T_\phi$ is well-defined (see Definition \ref{def:twist hodge lattice by endomorphism}).
    \item[ii)] Let $T'$ and $T''$ be integral irreducible Hodge lattices of K3-type, and suppose that $f\colon T\simeq T'$, $g\colon T\simeq T''$ are two isomorphisms of Hodge structures. Then we say that $f$ and $g$ are \textit{equivalent}, denoted $f\sim g$, if the composition $g\circ f^{-1}\colon T'\simeq T''$ is a Hodge isometry. It is easy to check that $\sim$ is an equivalence relation.
    \item[iii)] We denote $$H(T)\coloneqq \left(\bigcup_{T'}\operatorname{Isom}(T,T')\right)/\sim,$$ where the union runs over all integral, irreducible Hodge lattices of K3-type, and the set $\operatorname{Isom}(T,T')$ denotes the set of isomorphisms of Hodge structures $T\simeq T'$.
    \item[iv)] We write $I(T)$ for the set of irreducible Hodge lattices $T'$, that are isomorphic to $T$ as Hodge structures, considered up to Hodge isometries.
\end{enumerate} 
\end{definition}

\begin{remark}\label{rem: F H actions}
    Keeping the notation of Definition \ref{def: F H}, the group of Hodge automorphisms $\Aut(T)$ acts naturally on $F(T)$ and $H(T)$ as follows:
    \begin{itemize}
        \item[i)] For $h\in \Aut(T)$ and $\phi\in F(T)$, we define $h*\phi\coloneqq h^*\circ \phi\circ g$.
        \item[ii)] For $h\in \Aut(T)$ and $f\in H(T)$, we define $h*f\coloneqq f\circ h$. Note that if $f\colon T\simeq T'$ and $g\colon T\simeq T''$ are equivalent, then $h*f$ and $h*g$ are also equivalent. Indeed, $$h*g\circ (h*f)^{-1} = g\circ h\circ h^{-1}\circ f^{-1} = g\circ f^{-1}$$ is a Hodge isometry by assumption.
    \end{itemize}
\end{remark}
\begin{remark}\label{rem: quotient}
    For an integral, irreducible Hodge lattice $T$ of K3-type, it is easy to see that the map 
    $$\functionstar{H(T)}{I(T)}{\left(f\colon T\simeq T'\right)}{T'}$$ is $\Aut(T)$-invariant and induces a bijection $H(T)/\Aut(T)\simeq I(T).$
\end{remark}

The main technical result of this section is the following.

\begin{proposition} \label{prop: F H bijection}
    Let $T$ be an integral, irreducible Hodge lattice of K3-type. Then we have a bijection
    \begin{equation}\label{eq: F H bijection}
        \isomorphismstar{F(T)}{H(T)}{\phi}{\left(\id_{T}\colon T\simeq T_\phi\right).}
    \end{equation}
    Moreover, \eqref{eq: F H bijection} is equivariant for the $\Aut(T)$-actions on $F(T)$ and $H(T)$ of Remark \ref{rem: F H actions}.
\end{proposition}

Before we proceed to the proof of Proposition \ref{prop: F H bijection}, we need one further technical result, which essentially shows how to construct the inverse to \eqref{eq: F H bijection}.

\begin{lemma}\label{lem: isomorphic hodge structure means twisted by endomorphism}
    Let $T_1$ and $T_2$ be two integral irreducible Hodge lattices of K3-type that are isomorphic as Hodge structures. Let $f\colon T_1 \simeq T_2$ be an isomorphism of Hodge structures, and denote $$\phi\coloneqq i_{1,\Q}^{-1}\circ f^*_\Q \circ i_{2,\Q}\circ f_\Q \in \End((T_1)_\Q),$$ where $i_1\colon T_1\hookrightarrow T_1^*$ and $i_2\colon T_2\hookrightarrow T_2^*$ are the natural inclusions. Then $\phi$ is a rational Hodge endomorphism contained in $K(T_1)$, which maps $T_1$ into $T_1^*$. Moreover, $T_2$ is Hodge isometric to the Hodge lattice $(T_1)_\phi$ of Definition \ref{def:twist hodge lattice by endomorphism}.
\end{lemma}
\begin{proof}
    As discussed in Remark \ref{rem: isomorphism diagram that commutes iff isometry}, the corresponding diagram \begin{equation*}
        \xymatrix{
            T_1 \ar[r]^{f}_\simeq \ar@{^(->}[d]_{i_{1}} & T_2 \ar@{^(->}[d]^{i_{2}} \\ 
            T_1^* & T_2^*\ar[l]^{f^*}_\simeq
        }
    \end{equation*}
    is not necessarily commutative. 
    However, for $\phi\coloneqq i_{1,\Q}^{-1}\circ f^*_\Q \circ i_{2,\Q}\circ f_\Q,$ we have 
    \begin{equation}\label{eq: lattice equality}
        (\phi(x),-)_{T_1} = (f(x),f(-))_{T_2}
    \end{equation}
    for all $x\in (T_1)_\Q$. 
    By symmetry, we have $(x,\phi(-))_{T_1} = (\phi(x),-)_{T_1},$ hence $\phi\in K(T_1)$. Finally, since $(\phi(x),y)_{T_1} = (f(x),f(y))_{T_2}\in \Z$  for all $x,y\in T_1$, it follows that $\phi$ maps $T_1$ into $T_1^*$ by Remark \ref{rem: basics of twisting by endomorphism} ii).
    By construction, the diagram 
    \begin{equation*}
        \xymatrix{
            T_1 \ar[r]^{f}_\simeq \ar@{^(->}[d]_{\phi|_{T_1}} & T_2 \ar@{^(->}[d]^{i_{2}} \\ 
            T_1^* & T_2^*\ar[l]^{f^*}_\simeq
        }
    \end{equation*}
    commutes.
    Therefore, by Remark \ref{rem: isomorphism diagram that commutes iff isometry}, $T_2$ is Hodge isometric to $(T_1)_\phi$. 
\end{proof}

\begin{proof}[Proof of Proposition \ref{prop: F H bijection}]
    We first check the surjectivity of \eqref{eq: F H bijection}. Let $f\colon T\simeq T'$ be an isomorphism of Hodge structures, where $T'$ is some integral, irreducible Hodge lattice of K3-type. It follows from the proof of Lemma \ref{lem: isomorphic hodge structure means twisted by endomorphism} that $f$ is equivalent to $\id_T\colon T\simeq T_\phi$, where $\phi\coloneqq i_{T,\Q}^{-1}\circ f^*_\Q \circ i_{{T'},\Q}\circ f_\Q$, as in the proof of Lemma \ref{lem: isomorphic hodge structure means twisted by endomorphism}. Therefore \eqref{eq: F H bijection} is surjective.

    For injectivity, suppose $\id_T\colon T\simeq T_\phi$ is equivalent to $\id_T\colon T\simeq T_\psi$ for some $\phi,\psi\in F(T)$. By assumption, the diagram 
    \begin{equation*}
        \xymatrix{
            T_\phi \ar[d]_{\phi} \ar[r]^{\id_T}& T_\psi \ar[d]^{\psi} \\
            T_{\phi}^* & T_{\psi}^* \ar[l]^{\id_{T}^*}
        }
    \end{equation*}
    commutes, hence we have $\phi = \psi$, as required.

    Finally, to check the $\Aut(T)$-equivariance, let $h\in \Aut(T)$ be a Hodge automorphism of $T$ and let $\phi\in F(T)$. Note that we have the commutative diagram 
    \begin{equation}\label{eq: orbits are hodge isometric}
        \xymatrix{
        T_{h*\phi} \ar[r]^h \ar[d]_{(h*\phi)} & T_\phi \ar[d]^{\phi}\\
        T_{h*\phi}^*& T_\phi^*. \ar[l]^{h^*}
        }
    \end{equation}
    This shows that $\id_T\colon T\simeq T_{h*\phi}$ is equivalent to $h\colon T\simeq T_\phi$, hence \eqref{eq: F H bijection} is equivariant.
\end{proof}

\begin{corollary}\label{cor: main technical contribution}
    Let $T$ be an integral, irreducible Hodge lattice of K3-type. Let $\phi,\psi\in F(T)$. Then $T_\phi$ is Hodge isometric to $T_\psi$ if and only if there is a Hodge automorphism $h\in \Aut(T)$ such that $\phi = h*\psi$.
    In particular, we have a commutative diagram 
    \begin{equation*}\label{eq:main-maps}
    \xymatrix{
    F(T) \ar[r]^\simeq \ar[d]& H(T) \ar[d] & \\
    F(T)/\Aut(T)\ar[r]^\simeq& H(T)/\Aut(T) \ar[r]^(.65)\simeq & I(T).
    }
\end{equation*}    
\end{corollary}
\begin{proof}
    This follows immediately from Proposition \ref{prop: F H bijection} combined with Remark \ref{rem: quotient}.
\end{proof}

The main application of Proposition \ref{prop: F H bijection} is the following, which clarifies the relationship between $T$-equivalence and $D$-equivalence for sufficiently general K3 surfaces:

\begin{proposition}\label{prop: T general then T implies D}
    Let $X$ be a K3 surface with $\rho\neq 18$ and such that $\End(T(X))\simeq \Z$ (see Lemma \ref{lem: very general endomorphisms isometries}). Let $Y$ be a K3 surface which is $T$-equivalent to $X$ and which satisfies $\disc(T(X)) = \disc(T(Y))$. Then $X$ and $Y$ are $D$-equivalent. If $\rho = 18$, then $T(Y)$ is Hodge isometric to either $T(X)$ or $T(X)_{-1}$.
\end{proposition}
\begin{proof}
    Fix an isomorphism of Hodge structures $f\colon T(X)\simeq T(Y)$. By Proposition \ref{prop: F H bijection}, there is a rational endomorphism $\phi\colon T(X)_\Q\simeq T(X)_\Q$ such that $T(Y)$ is Hodge isometric to $T(X)_{\phi}$. By our assumption on $X$, $\phi$ is given by multiplication by some rational number $q\in \Q$. In particular, we have $$\disc(T(X)) = \disc(T(Y)) = \disc(T(X)_\phi) = q^{\rk T(X)}\cdot \disc(T(X)),$$ where the third equality follows from Lemma \ref{lem: discriminants of twisted hodge lattices}. This shows that $q = \pm 1$, hence $T(Y)$ is Hodge isometric to $T(X)$ or to $T(X)_{-1}$. If we assume $\rho\neq 18$, the signature of $T(X)$ is $(2,\rk(T(X))-2) \neq (2,2)$. Therefore, the signature of $T(X)_{-1}$ is not equal to that of $T(Y)$, hence $q = 1$, and $T(X)$ is Hodge isometric to $T(Y)$.
\end{proof}
\begin{remark}
    It is possible to find a K3 surface $X$ of Picard rank 18, and a K3 surface $Y$ which is $T$-equivalent to $X$ and which satisfies $\disc(T(X))=\disc(T(Y))$, yet $T(X)$ and $T(Y)$ are not Hodge isometric. For example, let $X$ be a K3 surface whose transcendental lattice is isometric to the lattice $$L \coloneqq A_2 \oplus \langle -2 \rangle \oplus \langle -2 \rangle,$$ where $A_2$ is the positive-definite lattice of rank $2$ whose matrix is given by 
    \[
        \left(
            \begin{matrix}
                2 & -1 \\ -1 & 2
            \end{matrix}
        \right).
    \]
    We claim that $L \not \simeq L(-1)$. Firstly, it is not difficult to check that $A_{A_2}(-1)\not\simeq A_{A_2}$. Since $A_L \simeq A_{A_2} \oplus A_{\langle -2 \rangle}^{\oplus 2},$ the $3$-primary part of $A_L$ is precisely $A_L^{(3)} \simeq A_{A_2}$. If we assume that $L \simeq L(-1)$, then it follows that $A_{A_2}\simeq A_{L}^{(3)} \simeq A_{L}(-1)^{(3)} \simeq A_{A_2}(-1)$, a contradiction. 

    Since there exists a primitive embedding $T(X)(-1) \hookrightarrow \klat$ by \cite[Corollary 1.12.3]{Nik80} (see also Example \ref{ex:Shioda--Inose counterexample}(ii)), it follows from the surjectivity of the period map \cite{Tod80} that there exists a K3 surface $Y$ with a Hodge isometry $T(Y) \simeq T(X)(-1)$. The K3 surface $Y$ is clearly $T$-equivalent to $X$, and we also have $\disc(T(Y)) = \disc(L(-1)) = \disc(L) = \disc(T(X))$, but $T(Y)$ is not isometric to $T(X)$, let alone Hodge isometric to $T(X)$. 

    Of course, we can replace $L$ by any even lattice $N$ of rank $4$ and signature $(2,2)$ to produce more examples of this phenomenon, provided that we have $N\not \simeq N(-1)$. 
\end{remark}
\section{Hodge realisation and discriminants}
\label{sec: L equivalence and D equivalence}
Following Efimov \cite{Efi18}, we denote by $\operatorname{HS}_{\Z,2}$ the additive category of integral, polarisable Hodge structures of weight 2. The \textit{Grothendieck group of polarisable Hodge structures of weight 2}, denoted $\gghs$, is defined to be the free abelian group generated by isomorphism classes of objects in $\operatorname{HS}_{\Z,2}$, modulo the relations of the form
\[
[H_1\oplus H_2] = [H_1] + [H_2] \qquad \text{for } H_1,H_2\in \operatorname{HS}_{\Z,2}.
\]

There is a group homomorphism, called the \textit{Hodge realisation map} $$\fonction{\operatorname{Hdg}_\Z}{\gravc}{\gghs}{[X]}{[H^2(X,\Z)]}$$ that factors through the localisation $\gravc[\LL^{-1}]$. In other words, for any two $L$-equivalent complex, smooth, projective varieties $X$, $Y$, we have $$[H^2(X,\Z)] = [H^2(Y,\Z)] \in \gghs.$$

For $H\in \operatorname{HS}_{\Z,2}$ an integral, polarisable Hodge structure of weight 2, we denote $\NS(H)\coloneqq H^{1,1}\cap H$.
We denote by $T(H)$ the minimal integral sub-Hodge structure of $H$ for which $H^{2,0}\subset T(H)_\CC$, and we call $T(H)$ the \textit{transcendental sub-Hodge structure} of $H$. Explicitly, $T(H)$ can be constructed as follows. The rational Hodge structure $H_\Q$ can be written as a direct sum of simple rational Hodge structures $H_\Q \simeq \bigoplus_{i = 1}^n V_i$, for example by \cite[\S II.7, Lemma 7.26]{Voi02}. Then we define $T(H)_\Q$ to be the direct sum of those $V_i$ which intersect $H^{2,0}$ non-trivially, and then we define the transcendental sub-Hodge structure of $H$ to be $T(H)\coloneqq T(H)_\Q\cap H$.  
For any two Hodge structures $H_1,H_2\in \operatorname{HS}_{\Z,2}$, it follows from the construction of $T(H)$ that we have $T(H_1\oplus H_2) = T(H_1)\oplus T(H_2)$.
Similarly, we have $\NS(H_1\oplus H_2)  = \NS(H_1) \oplus \NS(H_2)$.

Finally, we denote the \textit{gluing group} of $H$ by $$G(H)\coloneqq \frac{H}{\NS(H)\oplus T(H)}.$$
By the above discussion, it follows that we have $G(H_1\oplus H_2) \simeq G(H_1)\oplus G(H_2)$ for any $H_1,H_2\in \operatorname{HS}_{\Z,2}$. Therefore, we obtain a group homomorphism $$\fonction{D}{\gghs}{\Q^\times}{[H]}{|G(H)|.}$$

If $X$ is a K3 surface, then $H^2(X,\Z)$ is unimodular, hence the natural map $$H^2(X,\Z)/\big(\NS(X)\oplus T(X)\big) \simeq A_{T(X)}$$ is an isomorphism, where $A_{T(X)}= T(X)^*/T(X)$ is the discriminant group of $T$ \cite[Proposition 1.5.1]{Nik80}.

As an immediate consequence of the above discussion, we obtain:

\begin{lemma}\label{lem:L equivalent same discriminant}
    Let $X$ and $Y$ be $L$-equivalent K3 surfaces. Then $\disc(T(X)) = \disc(T(Y))$.
\end{lemma}

Our main result is the following.

\begin{theorem} \label{thm: L equivalent implies D equivalent in general}
    Let $X$ and $Y$ be K3 surfaces of Picard rank $\rho\neq 18$. Assume $\End(T(X))\simeq \Z$. If $X$ and $Y$ are $L$-equivalent, then they are $D$-equivalent.
\end{theorem}
\begin{proof}
    By \cite[Lemma 3.7]{Efi18}, $X$ and $Y$ are $T$-equivalent. Moreover, we have $\disc(T(X)) = \disc(T(Y))$ by Lemma \ref{lem:L equivalent same discriminant}. It follows from Proposition \ref{prop: T general then T implies D} that $X$ and $Y$ are $D$-equivalent, as required.
\end{proof}

Theorem \ref{thm: L equivalent implies D equivalent in general} applies to very general K3 surfaces of Picard rank $\rho\leq 17$ by Lemma \ref{lem: very general endomorphisms isometries}.

\section{Applications and challenges}
\label{sec: applications and challenges}
We now apply the results of the previous sections to study specific examples of K3 surfaces. Following this, we study higher-dimensional hyperk\"ahler manifolds of \Kn-type.

\begin{corollary}
    Let $X$ be a K3 surface of Picard rank 2 which admits an elliptic fibration $X\to \PP^1$ of multisection index $5$. Assume that $\End(T(X))\simeq \Z$. Let $Y$ be another K3 surface. Then $X$ and $Y$ are $L$-equivalent if and only if they are $D$-equivalent.
\end{corollary}
\begin{proof}
    By \cite[Corollary 5.13(i)]{MS24}, every K3 surface $Y$ such that $\Db(X)\simeq \Db(Y)$ admits an elliptic fibration $Y\to \PP^1$ such that $X\simeq \Jac^k(Y)$ for some $k = 1,\ldots, 4$. By \cite[Theorem 3.2]{SZ20}, $X$ and $Y$ are then $L$-equivalent. Conversely, if $X$ and $Y$ are $L$-equivalent, it follows from Theorem \ref{thm: L equivalent implies D equivalent in general} that $X$ and $Y$ are $D$-equivalent.
\end{proof}

Finally, we turn our attention to higher-dimensional hyperk\"ahler manifolds.
If $X$ is a hyperk\"ahler manifold of \Kn-type for some $n\geq 2$, then $H^2(X,\Z)$ is not unimodular and the natural map $$H^2(X,\Z)/\big(\NS(X)\oplus T(X)\big)\hookrightarrow A_{T(X)}$$ is injective but not necessarily surjective. This means that our proof of Lemma \ref{lem:L equivalent same discriminant} does not generalise to this setting. Nevertheless, Conjecture \ref{conj: L implies D} predicts that Lemma \ref{lem:L equivalent same discriminant} also holds for higher-dimensional hyperk\"ahler manifolds of \Kn-type. In fact, proving the analogue of Lemma \ref{lem:L equivalent same discriminant} for hyperk\"ahler manifolds of \Kn-type would be a big step towards proving Conjecture \ref{conj: L implies D}. We refer to \cite{Mei25} for details.

\begin{theorem}\label{thm: HK fourfolds}
    Let $X$ be a hyperk\"ahler fourfold of \Ktwo-type, and suppose $\End(T(X))\simeq \Z$. Assume that $X$ has Picard rank 1, and let $H\in \NS(X)$ be an ample generator of $\NS(X)$ with $H^2 = 2g$. Let $Y$  be a hyperk\"ahler fourfold of \Ktwo-type which is $L$-equivalent to $X$ and such that $\disc(T(Y)) = \disc(T(X))$.
    \begin{enumerate}
        \item[i)]  Suppose $g\equiv 1\pmod 4$, or $8\mid d$, or $\dv(H)=2$. Then $Y$ is $D$-equivalent to $X$. 
        \item[ii)] If $g\not\equiv1\pmod 4$ and $8\nmid d$ and $\dv(H)\neq 2$, then $X$ and $Y$ are twisted derived equivalent.
    \end{enumerate}
\end{theorem}
\begin{proof}
    Analogously to the proof of Proposition \ref{prop: T general then T implies D}, one may use Proposition \ref{prop: F H bijection} to show that there is a Hodge isometry $T(X)\simeq T(Y)$. The result now follows from \cite[Theorem 5.1]{KK24}.
\end{proof}

\begin{theorem} \label{thm: unimodular moduli}
    Let $S$ be a K3 surface with $\rho\neq 18$, such that $\NS(S)$ is unimodular, and such that $\End(T(S))\simeq \Z$. Let $M$ be a smooth moduli space of sheaves on $S$ of dimension $2n\geq 2$. Let $X$ be a hyperk\"ahler manifold of \Kn-type that is $L$-equivalent to $M$ and such that $\disc(T(X)) = \disc(T(M))$. Then $X$ is birational to $M$. In particular, $X$ and $M$ are $D$-equivalent.
\end{theorem}
\begin{proof}
    We may again use Proposition \ref{prop: F H bijection} to conclude that there is a Hodge isometry $T(X)\simeq T(M)$. In particular, $X$ is birational to a moduli space of sheaves on $S$ by \cite[Proposition 4]{Add16}. Therefore, by \cite[Proposition 5.11]{MM24}, it follows from the unimodularity of $\NS(S)$ that $X$ and $M$ are birational. Thus $X$ and $M$ are $D$-equivalent by \cite[Theorem 4.5.1]{Hal21}.
\end{proof}

\printbibliography
\end{document}